\documentclass[12pt]{amsart}
\usepackage[cp1251]{inputenc}
\usepackage[english,russian]{babel}
\usepackage{amsmath}
\usepackage{amssymb}
\usepackage{amsfonts, dsfont, graphicx, srcltx}

\usepackage{enumerate}

\def\udcs{517.574 : 517.576 : 517.550.4 : 517.547.2 : 517.518.244} 

\usepackage{enumerate}

\usepackage[pdftex,unicode,colorlinks,linkcolor=blue,
citecolor=red,bookmarksopen,pdfhighlight=/N]{hyperref}

\newtheorem{theorem}{Теорема}

\newtheorem{proposition}{Предложение}
\newtheorem*{ThL1}{Теорема 
B}
\newtheorem*{ThL}{Теорема Лиувилля}
\newtheorem*{mainlemma}{Основная лемма}
\newtheorem*{mainth}{Основная теорема}

\theoremstyle{definition}

\newtheorem*{rem}{Замечание}

\newcommand{\RR}{\mathbb{R}} 
\newcommand{\CC}{\mathbb{C}} 
\newcommand{\NN}{\mathbb{N}} 
 
\newcommand{\oB}{{\overline B}}

\DeclareMathOperator{\sbh}{{\sf sbh}}

\DeclareMathOperator{\ord}{{\sf ord}}

\DeclareMathOperator{\dd}{{\rm\,d\!}}

\begin{document}
УДК \udcs

\title[Теоремы типа Лиувилля для функций конечного порядка]
{Теоремы типа Лиувилля\\ вне малых исключительных множеств\\ для функций конечного порядка} 

\author{\large Б.Н.~Хабибуллин }
\address{Булат Нурмиевич Хабибуллин
\newline\hphantom{iii} Башкирский государственный университет,
 г.Уфа, Башкортостан, Россия}
\email{khabib-bulat@mail.ru}

\thanks{\copyright \ 2020 Хабибуллин Б.Н.}
\thanks{\rm Работа  выполнена в рамках реализации программы развития Научно-образовательного математического  центра Приволжского федерального округа, дополнительное соглашение № 075-02-2020-1421/1 к соглашению № 075-02-2020-1421.}
\thanks{\it Поступила 3 сентября 2020 г.}

\maketitle 
{
\small
\begin{quote}
\noindent{\bf Аннотация. } 
 Доказано, что выпуклые функции на вещественной прямой  $\RR$ и  субгармонические функции на $\RR^m$, $m>1$, конечного порядка, ограниченные сверху вне некоторого множества нулевой относительной лебеговой плотности, ограничены сверху всюду соответственно на $\RR$ или $\RR^m$. Отсюда субгармонические на комплексной плоскости $\CC$, целые и  плюрисубгармонические на $\CC^n$,  а также выпуклые   или  гармонические функции на $\RR^m$ конечного порядка, ограниченные сверху вне некоторого множества нулевой относительной лебеговой плотности, постоянны.    
\medskip

\noindent{\bf Ключевые слова:}{ целая функция, субгармоническая функция, плюрисубгармоническая функция, выпуклая функция, гармоническая функция, среднее по шару, функция конечного порядка, теорема Лиувилля}
\medskip 

\end{quote}
}

Основа нашей заметки  ---  классическая для {\it целых,\/} т.\,е. голоморфных  на {\it комплексной плоскости\/} $\CC$ или  на $\CC^n$, где $n\in \NN:=\{1,2, \dots\}$,  функций    
\begin{ThL} Если целая  функция ограничена, то она постоянна.   
\end{ThL}
Такое же заключение верно и для {\it ограниченных сверху\/} субгармонических функций на $\CC$ \cite[следствие 2.3.4]{Rans} и, как очевидное  следствие,  плюрисубгармонических функций на $\CC^n$,  выпуклых функций на {\it вещественной прямой\/} $\RR$ и, как мгновенное следствие,  на $\RR^m$ при $1<m\in \NN$, а также гармонических функций на $\RR^m$ при любых $m\in \NN$ \cite[теорема 1.19]{HK}.

Недавно в работе \cite[лемма 4.2]{BarBelBor18} была дана  версия теоремы Лиувилля для  целых функций {\it конечного порядка\/} на $\CC$, ограниченных не всюду, а лишь вне некоторого малого множества $E\subset \CC$.
 В  \cite[лемма 4.2]{BarBelBor18a} её доказательство откорректировано, а перед её формулировкой в  \cite[преамбула теоремы 2.1]{AlemBarBelHed20} отмечается, что установлена она А.\,А.~Боричевым. 
Приведённые в  \cite{BarBelBor18} и \cite{BarBelBor18a} доказательства используют далеко не тривиальные факты и рассуждения теорий функций комплексного переменного.

\begin{ThL1}[{\rm (\cite[лемма 4.2]{BarBelBor18},  \cite[лемма 4.2]{BarBelBor18a},  \cite[теорема 2.1]{AlemBarBelHed20})}]  Если целая функция   конечного порядка  на $\CC$ ограничена вне  
множества $E\subset \CC$ нулевой плоской плотности по плоской мере Лебега $\lambda$ в том смысле, что определён предел
\begin{equation}\label{m00}
\lim_{r\to +\infty} \frac{\lambda \bigl(\{z\in E\colon |z|\leq r\}\bigr)}{r^2}=0, 
\end{equation}
то эта функция  постоянная.
\end{ThL1}

Основные результаты настоящей статьи развивают  и распространяют теорему B  на плюрисубгармонические и целые  функции на $\CC^n$ для всех  $n\in \NN$, а также  на выпуклые и гармонические функции на $\RR^m$.  При этом наше доказательство и для случая целых функций одной комплексной переменной проще и построено на подходе,
отличном от применявшегося в предшествующем доказательстве А.А. Боричева теоремы B. Так, это <<одномерное>> доказательство уже было представлено  в краткой заметке \cite{Kha20ar} с описанием его перспектив. 

Пусть функция $M$ со значениями в  {\it расширенной вещественной прямой\/} $\overline \RR:=\RR\cup\{\pm\infty\}$ определена   на положительной полуоси $\RR^+:=\{x\in \RR\colon x\geq 0\}$, 
  в $\RR^m$ или в $\CC^n$, отождествляемом с $\RR^{2n}$, с евклидовой нормой $|\cdot|$, но, вообще говоря,  вне некоторого замкнутого шара $\oB(r)$ ограниченного радиуса $r\in \RR^+$ с центром в нуле.   {\it Порядок\/}  функции $M$ (около $\infty$)  можно определить как
\begin{equation}\label{ord}
\ord[f]:=\limsup_{|x|\to \infty}\frac{\ln \bigl(1+M^+(x)\bigr)}{\ln |x|}\in \RR^+\cup \{+\infty\}, 
\end{equation} 
где $M^+\colon x\mapsto \max\{0, M(x)\}$ --- положительная часть функции $M$. Порядок целой функции $f$ на $\CC^n$ определяется как порядок $\ord\bigl[\ln |f|\bigr]$ {\it плюрисубгармонической\/} функции $\ln|f|$.  

{\it Относительной лебеговой плотностью\/} измеримого по мере Лебега $\lambda$ на $\RR^m$ подмножества $E\subset \RR^m$ называем величину  
\begin{equation}\label{mRs}
{\sf L}_m (E):=\lim_{r\to +\infty} \frac{\lambda \bigl(E\cap B(r)\bigr)}{r^{m}}, 
\end{equation}
если предел существует. Определение очевидным образом переносится на  $\CC^n$, отождествлённое с $\RR^{2n}$, как 
${\sf L}_{2n} (E)$.

\begin{mainth} Пусть $m\in \NN$ и  $E\subset \RR^m$ --- множество  нулевой относительной лебеговой плотности\/  ${\sf L}_m(E)=0$ в $\RR^m$.  Если субгармоническая  функция $v$ конечного порядка на $\RR^m$ ограничена  сверху  на $\RR^m\!\setminus\!E$, то 
\begin{equation}\label{s}
\sup_{\RR^m}v=\sup_{\RR^m\!\setminus\!E}v<+\infty.   
\end{equation}
\end{mainth}
 
Пусть $n\in \NN$. Функция на $\CC^n$ называется плюрисубгармонической, если её сужение на каждую комплексную прямую --- субгармоническая.    В частности, при $n=1$ эти понятия --- одно и то же, а каждая плюрисубгармоническая функция на $\CC^n$ и субгармоническая   на $\RR^{2n}$. По  нашей основной теореме  из  классических теорем Лиувилля для плюрисубгармонических и целых функций сразу следует 

\begin{theorem}\label{th1} Пусть $n\in \NN$  и $E\subset \CC^n$ --- множество  нулевой относительной лебеговой плотности в $\CC^n$ в  смысле \eqref{mRs} на $\RR^{2n}$, отождествлённом с $\CC^n$, т.\,е. ${\sf L}_{2n}(E)=0$.
Если плюрисубгармоническая или целая функция  конечного порядка на $\CC^n$ ограничена  сверху  на $\CC^n\!\setminus\!E$, то она  постоянная. 
\end{theorem}
Субгармонические функции на  $\RR$ --- это в точности выпуклые функции, а при $m\in \NN$ каждая выпуклая или гармоническая функция субгармоническая.  По  нашей основной теореме  из  классических теорем Лиувилля для выпуклых или  гармонических функций на $\RR^m$  сразу следует  
\begin{theorem}\label{th2} Пусть $m\in \NN$ и  $E\subset \RR^m$ --- множество  нулевой относительной лебеговой плотности в $\RR^m$. Если выпуклая или гармоническая функция  конечного порядка на $\RR^m$ ограничена  сверху  на $\RR^m\!\setminus\!E$, то она  постоянная. 
\end{theorem}

Таким образом, достаточно доказать основную теорему, к чему и переходим. 

Для $m\in \NN$,   $x\in \RR^m$ и $r\in \RR^+$ через   $\overline  B(x,r):=\{x' \in \RR^m \colon |x'-x|\leq r\}$ 
обозначаем {\it замкнутый шар\/} в $\RR^m$ {\it радиуса $r$ с центром $x$,\/} и, 
 как и прежде, $\overline  B(r):=\overline B(0,r)$.  Аналогично для $\CC^n$, отождествляемом с $\RR^{2n}$.  
Для  $\lambda$-интегрируемой функции $v\colon \overline B(x,r)\to  \overline \RR$   полагаем
\begin{equation}\label{Bv}
{\sf B}_v(x,r):=\frac{1}{\lambda\bigl(\overline B(x,r)\bigr)}\int_{\overline B(x,r)} v \dd \lambda=\frac{1}{b_mr^m}\int_{\overline B(x,r)} v \dd \lambda, \quad  {\sf B}_v(r):={\sf B}_v(0, r),
\end{equation}
где  $b_m$ --- {\it объём единичного шара.\/}  Это соответственно {\it средние функции $v$ по замкнутым шарам\/} $\overline B(x,r)$ и $\overline B(r)$.  
{\it Положительность\/}  понимается как $\geq 0$, {\it отрицательность\/} --- это $\leq 0 $.  

\begin{proposition}\label{1} Пусть $v$ --- положительная $\lambda$-измеримая функция на замкнутом шаре  $\overline B(R)\subset \RR^m$, $0<r<R$,  и $x\in \oB (r)$. Тогда
\begin{equation}\label{BB}
{\sf B}_v(x,R-r)\leq \Bigl(1+\frac{r}{R-r}\Bigr)^m  {\sf B}_v(R).
\end{equation} 
\end{proposition}
\begin{proof} По определению \eqref{Bv},  в силу положительности $v$ на $\oB(R)$ и включений $\oB (x,R-r)\subset 
\oB(R)$ для всех $x\in \oB(r)$ получаем
\begin{multline*}
{\sf B}_v(x,R-r)\overset{\eqref{Bv}}{=}\frac{1}{b_m(R-r)^m}\int_{\overline B(x,R-r)} v \dd \lambda
\leq \frac{1}{b_m(R-r)^m}\int_{\overline B(R)} v \dd \lambda\\
=\frac{b_mR^m}{b_m(R-r)^m}\frac{1}{b_mR^m}\int_{\overline B(R)} v \dd \lambda
\overset{\eqref{Bv}}{=} \Bigl(1+\frac{r}{R-r}\Bigr)^m  {\sf B}_v(R),
\end{multline*}
что и требовалось для \eqref{BB}.
\end{proof}

Через $\sbh(S)$ обозначаем класс всех {\it субгармонических\/} ({\it локально выпуклых} при $m=1$) {\it функций\/} на каких-либо открытых  окрестностях множества $S\subset \RR^m$.   
Роль средних по шару из \eqref{Bv}  для субгармонических функций обусловлена полностью характеризующим их, при условии полунепрерывности сверху  и локальной интегрируемости по мере Лебега $\lambda$, неравенством о среднем по шару \cite{Rans}, \cite{HK}: 
\begin{equation}\label{inav}
v(x)\leq {\sf B}_v(x,r) \quad\text{при $v\in \sbh\bigl(\oB(x,r)\bigr)$}.
\end{equation}

\begin{proposition}\label{2}
Пусть $v$ --- субгармоническая функция на замкнутом шаре $\overline B(R)\subset \RR^m$,  $r\in (0,R)$ и $E\subset \overline B(r)$ --- $\lambda$-измеримое множество. Тогда 
\begin{equation}\label{v+}
\int_{E}v\dd \lambda \leq 
\Bigl(1+\frac{r}{R-r}\Bigr)^m \lambda (E){\sf B}_{v^+}(R).
\end{equation}
\end{proposition}
\begin{proof} Из неравенства \eqref{inav} о среднем по шару получаем 
\begin{equation*}
v(x)\leq {\sf B}_{v}(x,R-r) \leq {\sf B}_{v^+}(x,R-r) \quad \text{для каждой точки $x\in \oB (r)$}.
\end{equation*}
Интегрирование крайних частей этого неравенства по мере Лебега $\lambda$ на множестве $E$ даёт неравенство 
\begin{equation*}
\int_{E}v\dd \lambda \leq \int_{E} {\sf B}_{v^+}(x,R-r) \dd \lambda (x).
\end{equation*}
Отсюда, по неравенству \eqref{BB} предложения \ref{1}, применённому к  подынтегральному выражению с положительной функцией $v^+$ в последнем интеграле, получаем 
\begin{equation*}
\int_{E}v\dd \lambda \leq \int_{E}\Bigl(1+\frac{r}{R-r}\Bigr)^m  {\sf B}_{v^+}(R)\dd \lambda (x)
=\Bigl(1+\frac{r}{R-r}\Bigr)^m {\sf B}_{v^+}(R)\lambda(E),
\end{equation*}  
что и даёт \eqref{v+}. 
\end{proof}

 \begin{mainlemma} Пусть $v$ --- субгармоническая функция на  шаре $\overline B(R)\subset \RR^m$. Тогда 
для любого числа $r\in (0,R)$ и для любого $\lambda$-измеримого подмножества $E\subset \overline B(r)$ имеет место неравенство
\begin{equation}\label{v}
{\sf B}_v(r)\leq \frac{1}{b_mr^m}\int_{\oB(r)\!\setminus\!E}v\dd \lambda
+\frac{1}{b_m}\Bigl(1+\frac{r}{R-r}\Bigr)^m\frac{\lambda (E)}{r^m}{\sf B}_{v^+}(R).
\end{equation}
 \end{mainlemma}
\begin{proof} По определению \eqref{Bv}
\begin{equation*}
{\sf B}_v(r)= \frac{1}{b_mr^m}\int_{B(r)\!\setminus\!E}v\dd \lambda
+\frac{1}{b_mr^m}\int_{E}v\dd \lambda,
\end{equation*}
откуда по неравенству \eqref{v+} предложения \ref{2},  применённому к последнему интегралу, получаем в точности  требуемое \eqref{v}.
\end{proof}
 
\begin{proof}[Доказательство основной теоремы]
Положим \begin{equation}\label{M}
 M:=\sup_{\RR^m\!\setminus\!E} v\in \RR. 
\end{equation}
По условию ограниченности сверху   на $\RR^m\!\setminus\!E$ функции $v$ 
можно рассмотреть субгармоническую функцию  $v-M$, отрицательную на $\RR^m\!\setminus\!E$.  Применим теперь основную лемму при произвольных $0<r\in \RR^+$ с $R=2r$ и  с множеством-пересечением $E\cap \oB (r)\subset \oB (r)$ в роли множества  $E$ к субгармонической функции $(v-M)^+\geq 0$, где первый интеграл в правой части \eqref{v} будет равен нулю, а в итоге получим
  \begin{multline*}
{\sf B}_{(v-M)^+}(r)\leq \frac{1}{b_m}\Bigl(1+\frac{r}{2r-r}\Bigr)^m\frac{\lambda \bigl(E\cap \oB(r)\bigr)}{r^m}{\sf B}_{(v-M)^+}(2r)\\
=\frac{2^m}{b_m}\frac{\lambda \bigl(E\cap \oB(r)\bigr)}{r^m}{\sf B}_{(v-M)^+}(2r) 
 \quad\text{при всех $0<r\in \RR^+$}.
\end{multline*}
Отсюда по условию 
${\sf L}_m (E)\overset{\eqref{mRs}}{=}0$ для функции 
\begin{equation}\label{Bm}
r\underset{\text{\tiny $0<r\in \RR^+$}}{\longmapsto} {\sf B}_{(v-M)^+}(r)\in \RR^+
\end{equation}
 имеем 
\begin{equation}\label{+}
{\sf B}_{(v-M)^+}(r)=o\bigl({\sf B}_{(v-M)^+}(2r)\bigr)
\quad \text{при $r\to+\infty$}.
\end{equation} 
Функция \eqref{Bm} конечного порядка $\ord[{\sf B}_{(v-M)^+}]\in \RR^+$, поскольку  $\ord[(v-M)^+]\in \RR^+$ ввиду
конечности порядка $\ord[v]$. Следовательно,  \eqref{+} возможно только в случае  ${\sf B}_{(v-M)^+}\equiv 0$, и, как следствие $(v-M)^+\equiv 0$.   Это вместе с \eqref{M} даёт  \eqref{s}. \end{proof}

\begin{rem}
Условие нулевой лебеговой плотности ${\sf L}_m(E)=0$  в основной теореме и в теореме \ref{th2}, как и то же самое 
 с $m:=2n$ в теореме \ref{th1}, можно заменить на формально   более слабое условие: {\it существует неограниченная последовательность положительных чисел $(r_k)_{k\in \NN}$, для которой
\begin{equation*}
\limsup_{k\to \infty} \frac{r_{k+1}}{r_k}<+\infty
\quad\text{и при этом}\quad
\lim_{k\to \infty} \frac{\lambda \bigl(E\cap B(r_k)\bigr)}{r_k^{m}}=0.
\end{equation*}
} 
Формально, поскольку последнее влечёт за собой  ${\sf L}_m(E)=0$. 
\end{rem}


\end{document}